\documentclass[reqno,11pt]{amsart}
\usepackage{amsmath,amsfonts,amssymb,amsthm}\usepackage[foot]{amsaddr}
\usepackage{mathtools,mathrsfs,yhmath}
\usepackage{graphicx,color,xcolor}
\usepackage{enumitem}
\usepackage{hyperref}
\usepackage[backend=biber, sorting=nty]{biblatex}

\bibliography{Mikabib.bib}

 \allowdisplaybreaks
\numberwithin{equation}{section}

\newcommand{\R}{\mathbb{R}}

\newcommand{\N}{\mathbb{N}}
\newcommand{\T}{\mathbb{T}}

\newcommand{\D}{\mathcal{D}}

\newcommand{\KW}{\mathrm{KW}}
\newcommand{\HW}{\mathrm{HW}}
\renewcommand{\L}{\mathrm{L}}
\newcommand{\X}{\mathcal{X}}
\newcommand{\Rd}{\R^d}
\newcommand{\Td}{\T^d}
\newcommand{\TdTd}{\T^d\times \T^d}
\newcommand{\TdRd}{\T^d\times \R^d}

\newcommand{\XRd}{\X \times \R^d}
\DeclarePairedDelimiter{\abs}{\lvert}{\rvert}
\DeclarePairedDelimiter{\norm}{\lVert}{\rVert}

\DeclareMathOperator{\Id}{Id}
\DeclareMathOperator{\divergence}{div}

\DeclareMathOperator{\loc}{loc}

\numberwithin{equation}{section}
\newtheorem{theorem}{Theorem}[section]
\newtheorem{lemma}[theorem]{Lemma}
\newtheorem{proposition}[theorem]{Proposition}

\newtheorem{definition}[theorem]{Definition}

\newtheorem{remark}[theorem]{Remark}

\makeatletter
\let\oldabs\abs
\def\abs{\@ifstar{\oldabs}{\oldabs*}}
\let\oldnorm\norm
\def\norm{\@ifstar{\oldnorm}{\oldnorm*}}
\makeatother

\begin{document}

\title[Stability estimates for the Vlasov-Poisson system in $p$-kinetic Wasserstein
distances]{Stability estimates for the Vlasov-Poisson system in $p$-kinetic Wasserstein
distances}

\author{Mikaela Iacobelli}
\address{ETH Z\"urich, Department of Mathematics, R\"amistrasse 101, 8092 Z\"urich, Switzerland} \email{mikaela.iacobelli@math.ethz.ch}

\author{Jonathan Junn\'e}
\address{TU Delft, Delft Institute of Applied Mathematics, Mekelweg 4, 2628 CD Delft, Netherlands}
\email{j.junne@tudelft.nl}

\subjclass[2020]{35Q83, 82C40, 82D10, 35B35}
\maketitle
\begin{abstract}
  We extend Loeper's $L^2$-estimate \cite[Theorem 2.9]{Loep} relating the force fields to the densities for the Vlasov-Poisson system to $L^p$, with $1 < p <+\infty$, based on the Helmholtz-Weyl decomposition. This allows us to generalize both the classical Loeper's $2$-Wasserstein stability estimate \cite[Theorem 1.2]{Loep} and the recent stability estimate by the first author relying on the newly introduced kinetic Wasserstein distance \cite[Theorem 3.1]{Iacobelli2022} to kinetic Wasserstein distances of order $1 <p<+\infty$.
\end{abstract}

\section{Introduction}

\subsection{General overview}
Monge-Kantorovich distances, also known as Wasserstein distances, are used extensively in kinetic theory, in particularly in the context of stability, convergence to equilibrium and mean-field limits.
A first celebrated result for the $1$-Monge-Kantorovich distance is due to Dobrushin \cite[Theorem 1]{Dob}, who proved the well-posedness for Vlasov equations with $C^{1,1}$ potentials. An explanation of Dobrushin’s stability estimate and its consequences on the mean-field limit for the Vlasov equation can be found in \cite[Chapter 1]{Golse} and \cite[Chapter 3]{JabinReviewMFL}, and we refer to \cite[Section 3]{GPI-proc-wp} for a survey on well-posedness for the Vlasov-Poisson system.

Regarding the $2$-Wasserstein distance, Loeper proved \cite[Theorem 1.2]{Loep} a uniqueness criterion for solutions with bounded density based on a $2$-Wasserstein distance stability estimate using both a link between the $\dot{H}^{-1}$-seminorm and the $2$-Wasserstein distance, and the fact that the Coulomb kernel is generated by a potential solving the Poisson equation. In addition to the Vlasov-Poisson system, this criterion gives a new proof of uniqueness \textit{\`a la} Yudovich for $2D$ Euler. Beyond bounded density, Loeper's uniqueness criterion has been extended for some suitable Orlicz spaces using the 1-Monge-Kantorovich distance by Miot \cite[Theorem 1.1]{MiotCMP2016} and Miot, Holding \cite[Theorem 1.1]{HoldingMiot2018}. 

On the Torus, Loeper's criterion was improved by Han-Kwan, Iacobelli \cite[Theorem 3.1]{HKI2} for the Vlasov-Poisson system, and more recently for the Vlasov-Poisson system with massless electrons by Griffin-Pickering, Iacobelli \cite[Theorem 4.1]{GPI-WP}. 

The aim of this work is twofold. The first goal is to generalize Loeper's $2$-Wasserstein distance stability estimate to $p$-Wasserstein distances for $1 < p < +\infty$. The second goal is to extend the recent stability estimate \cite[Theorem 3.1]{Iacobelli2022} by the first author relying on the newly introduced kinetic Wasserstein distance \cite[Theorem 3.1]{Iacobelli2022} to kinetic Wasserstein distances of order $1 <p<+\infty$.

\subsection{Definitions and main results}
We first recall the classical Wasserstein distance (see \cite[Chapter 6]{Vil09}) on the product space $\X \times \Rd$, with $\X$ denoting in the sequel either the $d$-dimensional torus $\Td$ or the Euclidean space $\Rd$:
\begin{definition}\label{def:Wasserstein}
    Let $\mu, \nu$ be two probability measures on $\X \times \Rd$. The \emph{Wasserstein distance} of order $p$, with $p \ge 1$, between $\mu$ and $\nu$ is defined as
    \begin{equation*}
        W_p(\mu, \nu) := \left( \inf_{\pi \in \Pi(\mu, \nu)} \int_{(\X \times \Rd)^2} \abs{x - y}^p + \abs{v - w}^p \: d\pi(x,v,y,w) \right)^{1/p},
    \end{equation*}
    where $\Pi(\mu, \nu)$ is the set of \emph{couplings}; that is, the set of probability measures with \emph{marginals} $\mu$ and $\nu$. A coupling is said to be \emph{optimal} if it minimizes the Wasserstein distance.
\end{definition}

We consider two solutions $f_1, f_2$ of the Vlasov-Poisson system on $\X$, with either gravitational or electrostatic interaction encoded by $\sigma = \pm 1$, namely, 
\begin{equation}\label{eq:VP Rd}
    \partial_t f + v \cdot \partial_x f - \nabla U \cdot \nabla_v f = 0, \quad \sigma\Delta U := \rho_f - 1, \quad \rho_f := \int_{\Rd} f \: dv
\end{equation}
on the torus, and
\begin{equation}\label{eq:VP Td}
    \partial_t f + v \cdot \partial_x f - \nabla U \cdot \nabla_v f = 0, \quad \sigma\Delta U := \rho_f, \quad \rho_f := \int_{\Rd} f \: dv
\end{equation}
on the whole space, with initial profiles $f_1(0), f_2(0)$, and respective flows $Z_1 := (X_1, V_1)$ and $Z_2 := (X_2, V_2)$ satisfying the system of characteristics 
\begin{equation*}
    \dot{X} = V, \quad \dot{V} = -\nabla U, \quad X(0, x, v) = x, \quad V(0, x, v) = v.
\end{equation*}
The flows yield solutions $f_1(t) = Z_1(t, \cdot, \cdot)_\# f_1(0)$ and $f_2(t) = Z_2(t, \cdot, \cdot)_\# f_2(0)$ as pushforwards of the initial data.

\noindent{\bf (1) A new $L^p$-estimate for the difference of force fields}.
Loeper estimates the $L^2$-norm \cite[Theorem 2.9]{Loep} of the difference of force fields with the Wasserstein distance between the densities. We extend the $L^2$-estimate to $L^p$ for $1 < p < +\infty$ using the Helmholtz-Weyl decomposition of $L^p(\X) = G_p(\X) \oplus H_p(\X)$ into its hydrodynamic space that we recall (see \cite[Chapter III]{Galdi2011}):
\begin{definition}\label{def:Hydro spaces}
    The \emph{hydrodynamic spaces} are the closed subspaces of $L^p(\X)$ defined as
    \begin{equation*}
        G_p(\X) := \left\{ u \in L^p(\X); \quad u = \nabla w \text{ for some } w \in W^{1,p}_{\loc}(\X)\right\}
    \end{equation*}
    and
    \begin{equation*}
        H_p(\X) := \overline{\left\{u \in C_c^\infty(\X);\quad \divergence u = 0 \text{ on }\X\right\}}.
    \end{equation*}
\end{definition}
\begin{remark}\label{rmk:HW L^2}
    Note that this decomposition breaks down for $p = 1$ or $p = +\infty$ and does not hold for general domains in $L^p$. It is equivalent to the solvability of an Neumann problem (see \cite[Lemma III.1.2]{Galdi2011}), while an orthogonal decomposition in $L^2$ is always possible, whatever the domain is.
\end{remark}

In our setting; that is either on the torus or on the Euclidean space, the Helmholtz-Weyl decomposition holds \cite[Theorem III.1.1 \& Theorem III.1.2]{Galdi2011};
\begin{theorem}[Helmholtz-Weyl decomposition] \label{thm:HW}
    The Helmholtz-Weyl decomposition holds for $L^p(\X)$, for any $1 < p < +\infty$; that is,
    \begin{equation*}
        L^p(\X) = G_p(\X) \oplus H_p(\X).
    \end{equation*}
    Moreover, when $p = 2$, this decomposition is orthogonal.
\end{theorem}

The validity of the Helmholtz-Weyl decomposition implies the existence of an Helmholtz-Weyl bounded linear projection operator (see \cite[Remark III.1.1]{Galdi2011})
\begin{equation*}
    \displaystyle P_p : L^p(\X) \to H_p(\X)
\end{equation*}
with range $H_p(\X)$ and with $G_p(\X)$ as null space. More precisely, there is a constant $C_{P_p} > 0$ that only depends on $p$ and $\X$ such that for all $u \in L^p(\X)$, it holds
\begin{equation}\label{eq:HW proj}
    \displaystyle \norm{P_p(u)}_{L^p(\X)} \le C_{P_p}\norm{u}_{L^p(\X)}.
\end{equation}

Using optimal transport techniques, Loeper manages to link the strong dual homogeneous Sobolev norm and Wasserstein distances between densities, and we recall those notions:  
\begin{definition}\label{def:Homogeneous Sobolev}
    Let $1 < p < +\infty$. The \emph{homogeneous Sobolev space} is the space
    \begin{equation*}
        \dot{W}^{1,p}(\X) := \left\{[g];\quad g \in W^{1,p}_{\loc}(\X), \quad \nabla g \in L^p(\X) \right\},
    \end{equation*}
    where $[\cdot] := \{\cdot + c; \: c\in\R\}$ denotes the equivalence class of functions up to a constant, together with the norm
    \begin{equation*}
        \norm{[g]}_{\dot{W}^{1,p}(\X)} := \norm{\nabla g}_{L^p(\X)}.
    \end{equation*}
\end{definition}
This is a Banach space for which the equivalence classes of test functions
\begin{equation*}
    \dot{\D}(\X) := \{[\phi]; \quad \phi \in C^\infty_c(\X)\}
\end{equation*}
are dense in it (see \cite[Theorem 2.1]{Ortner2012}). \begin{definition}\label{def:Dual Homogeneous Sobolev}
    We define the \emph{dual homogeneous Sobolev space} $\dot{W}^{-1,p}(\X)$ to be the topological dual of $\dot{W}^{1,p'}(\X)$ equipped with the \emph{strong dual homogeneous Sobolev norm}. For a function $h$ with $\int h = 0$, by density,
    \begin{align*}
        \norm{h}_{\dot{W}^{-1, p}(\X)} &:= \sup\left\{\int_\X h\,[g] \: dx; \quad g \in \dot{W}^{1,p'}(\X), \quad \norm{[g]}_{\dot{W}^{1,p'}(\X)} \le 1 \right\} \\
            &= \sup\left\{\int_\X h\,[\phi] \: dx; \quad \phi \in \dot{\D}(\X),\quad \norm{[\phi]}_{\dot{W}^{1,p'}(\X) \le 1} \right\}.
    \end{align*}
\end{definition}

First, we extend this connection for densities to $L^p$. Using the machinery of Helmholtz-Weyl decomposition, we generalize \cite[Lemma 2.10]{Loep} into the following:
\begin{lemma}\label{thm:Delta U}
    Let $\rho_1, \rho_2 \in L^\infty(\X)$ be two probability measures, and let $U_i$ satisfy $\sigma\Delta U_i = \rho_i$ for $\X = \Rd$, or $\sigma\Delta U_i = \rho_i - 1$ for $\X = \Td$, with $i = 1,2, \sigma = \pm 1$, in the distributional sense. Let $1 < p < +\infty$. Then there is a constant $C_{\HW} > 0$ that only depends on $p$ and $\X$ such that
    \begin{equation}\label{eq:L^p dot W^-1,p estimate}
         \norm{\nabla U_1 - \nabla U_2}_{L^p(\X)} \le C_{\HW}\norm{\rho_1 - \rho_2}_{\dot{W}^{-1,p}(\X)}.
    \end{equation}
\end{lemma}

Second, we adapt Loeper's argument of the $L^2$-estimate \cite[Theorem 2.9]{Loep}(see also \cite[Proposition  1.1]{Santambrogio23} in bounded  convex domains) relating negative homogeneous Sobolev norms to Wasserstein distances with this new link on force fields to get the new $L^p$-estimate allowing us to generalize stability estimates;
\begin{proposition}\label{thm:L^p estimate}
    Let $\rho_1, \rho_2 \in L^\infty(\X)$ be two probability measures, and let $U_i$ satisfy $\sigma\Delta U_i = \rho_i$ for $\X = \Rd$, or $\sigma\Delta U_i = \rho_i - 1$ for $\X = \Td$, with $i = 1,2,$ $\sigma = \pm 1$, in the distributional sense. Let $1 < p < +\infty$. Then there is a constant $C_{\HW} > 0$ that only depends on $p$ and $\X$ such that
    \begin{equation}\label{eq:L^p force field estimate}
         \norm{\nabla U_1 - \nabla U_2}_{L^p(\X)} \le C_{\HW} \max \left\{ \norm{\rho_1}_{L^\infty(\X)}, \norm{\rho_2}_{L^\infty(\X)} \right\}^{1/p'}W_p(\rho_1, \rho_2).
    \end{equation}
\end{proposition}

\noindent{\bf (2) Loeper's stability estimate in $W_p$.}
Loeper noted \cite[Lemma 3.6]{Loep} that both the Wasserstein distance of order two of the solutions and of the associated densities are bounded by a flow quantity $Q$ given by
\begin{equation*}
    Q(t) := \int_{\Rd \times \Rd} \abs{X_1(t,x,v) - X_2(t,x,v)}^2 + \abs{V_1(t,x,v) - V_2(t,x,v)}^2 \: df^0(x, v), 
\end{equation*}
and the bounds read as
\begin{equation}\label{eq:WQ estimate}
    W_2^2(f_1(t), f_2(t)) \le Q(t), \quad W_2^2(\rho_{f_1}(t), \rho_{f_2}(t))  \le Q(t).
\end{equation}
Loeper uses the quantity $Q(t)$ together with the $L^2$-estimate on the force fields to prove the stability estimate \cite[Theorem 1.2]{Loep} leading to the uniqueness of weak solutions. By modifying the quantity $Q(t)$ to
\begin{align*}
    Q_p(t) &:= \int_{(\X \times \Rd)^2} \abs{X_1(t,x,v) - X_2(t,y,w)}^p + \abs{V_1(t,x,v) - V_2(t,y,w)}^p \: d\pi_0(x,v,y,w) \\
        &= \int_{(\X \times \Rd)^2} \abs{x-y}^p + \abs{v - w}^p \: d\pi_t(x,v,y,w),
\end{align*}
where $\pi_t \in \Pi(f_1(t), f_2(t))$ and $\pi_0$ is an optimal $W_p$ coupling (see \cite[Section 4]{GPI-WP} for a construction of $\pi_t$), we are able to generalise Loeper's stability estimate \cite[Theorem 1.2]{Loep}, and \cite[Theorem 3.1]{HKI2} both on the torus $\Td$ and on the whole space $\Rd$, to any Wasserstein distance of order $p$, with $ 1 < p < +\infty$;
\begin{theorem}\label{thm:Loeper's W_p}
    Let $f_1, f_2$ be two weak solutions to the Vlasov-Poisson system on $\X$ with respective densities
    \begin{equation*}
        \rho_{f_1} := \int_{\Rd} f_1 \: dv, \quad \rho_{f_2} := \int_{\Rd} f_2 \: dv. 
    \end{equation*}
    Let $1 < p < +\infty$, and set
    \begin{equation}\label{eq:A Loeper}
        A(t) := \norm{\rho_{f_2}(t)}_{L^\infty(\X)} + \norm{\rho_{f_1}(t)}_{L^\infty(\X)}^{1/p}\max \left\{ \norm{\rho_{f_1}(t)}_{L^\infty(\X)}, \norm{\rho_{f_2}(t)}_{L^\infty(\X)} \right\}^{1/p'},
    \end{equation}
    which is assumed to be in $L^1([0, T))$ for some $T > 0$. Then there is a constant $C_{\L} > 0$ that only depends on $p$ and $d$ such that if $W_p^p(f_1(0), f_2(0))$ is sufficiently small so that $W_p^p(f_1(0), f_2(0)) \le (4\sqrt{d}/e)^p$ and
    \begin{equation}\label{eq:condition on initial W_p}
        \abs{\log\left(\frac{W_p^p(f_1(0), f_2(0))}{\left(4\sqrt{d}\right)^p}\right)} \ge p\exp\left(-C_{\L}\int_0^T A(s) \: ds\right)
    \end{equation}
    then
    \begin{equation}\label{eq:Loeper's W_p stability estimate}
        W_p^p(f_1(t), f_2(t)) \le \left(4\sqrt{d}\right)^p \exp\left\{ \log\left(\frac{W_p^p(f_1(0), f_2(0))}{\left(4\sqrt{d}\right)^p}\right)\exp\left( C_{\L}\int_0^t A(s) \: ds\right)\right\}.
    \end{equation}
\end{theorem}

\noindent{\bf (3) An improved $W_p$ stability estimate via kinetic Wasserstein distance.}
Due to the anisotropy between position and momentum variables, we use an adapted Wasserstein distance designed for kinetic problems taking this into account as introduced in \cite[Section 4]{Iacobelli2022}:
\begin{definition}
    Let $\mu, \nu$ be two probability measures on $\X \times \Rd$. The \emph{kinetic Wasserstein distance} of order $p$, with $p \ge 1$, between $\mu$ and $\nu$ is defined as
    \begin{equation*}
        W_{\lambda, p}(\mu, \nu) := \left(\inf_{\pi \in \Pi(\mu, \nu)} D_p(\pi, \lambda)\right)^{1/p},
    \end{equation*}
    where $D_p(\pi, \lambda)$ is the unique number $s$ such that
    \begin{equation*}
        s - \lambda(s)\int_{(\X \times \Rd)^2} \abs{x - y}^p \: d\pi(x, y, v, w) = \int_{(\X \times \Rd)^2} \abs{v - w}^p \: d\pi(x, y, v, w),
    \end{equation*}
    with $\lambda : \R^+ \to \R^+$ a decreasing function.
\end{definition}
We consider the quantity $D_p(t)$ for $\pi_t$ and $\lambda(t) = \abs{\log D_p(t)}^{p/2}$ (see \cite[Lemma 3.7]{Iacobelli2022} for the proof of existence) given by
\begin{align*}
    D_p(t) &:= \frac{1}{p}\int_{(\X\times\Rd)^2} \lambda(t)\abs{X_1(t, x, v) - X_2(t, y, w)}^p + \abs{V_1(t, x, v) - V_2(t, y, w)}^p \: d\pi_0(x, v, y, w) \\
        &= \frac{1}{p} \int_{(\X\times\Rd)^2} \lambda(t)\abs{x - y}^p + \abs{v - w}^p \: d\pi_t(x, v, y, w).
\end{align*}
This quantity also compares to the usual Wasserstein distance $W_p$ as does $Q_p(t)$, and this allows us to generalize the recent Iacobelli's stability estimate \cite[Theorem 3.1]{Iacobelli2022} to the following:
\begin{theorem}\label{thm:Iacobelli's W_p}
    Let $f_1, f_2$ be two weak solutions to the Vlasov-Poisson system on $\X$ (\ref{eq:VP Td}) with respective densities
    \begin{equation*}
        \rho_{f_1} := \int_{\Rd} f_1 \: dv, \quad \rho_{f_2} := \int_{\Rd} f_2 \: dv. 
    \end{equation*}
    Let $1 < p < +\infty$, and set
    \begin{equation}\label{eq:A Iacobelli}
        A(t) := \norm{\rho_{f_2}(t)}_{L^\infty(\X)} + \norm{\rho_{f_1}(t)}_{L^\infty(\X)}^{1/p}\max \left\{ \norm{\rho_{f_1}(t)}_{L^\infty(\X)}, \norm{\rho_{f_2}(t)}_{L^\infty(\X)} \right\}^{1/p'},
    \end{equation}
    which is assumed to be in $L^1([0, T))$ for some $T > 0$. Then there is a universal constant $c_0 > 0$ and a constant $C_{\KW} > 0$ that depends only on $p$ and $d$ such that if $W_p^p(f_1(0), f_2(0))$ is sufficiently small so that $W_p^p(f_1(0), f_2(0)) \le pc_0$ and
    \begin{equation}\label{eq:kinetic condition on initial W_p}
        \sqrt{\abs{\log\Bigg\{W_p^p(f_1(0), f_2(0))\abs{\log\left(\frac{1}{p}W_p^p(f_1(0), f_2(0))\right)}\Bigg\}}} \ge C_{\KW}\int_0^T A(s) \: ds + 1,
    \end{equation}
    then
    \begin{multline}\label{eq:Iacobelli's W_p stability estimate}
        W_p^p(f_1(t), f_2(t)) \\ \le p\exp\Bigg\{-\left(\sqrt{\abs{\log\Bigg\{W_p^p(f_1(0), f_2(0))\abs{\log\left(\frac{1}{p}W_p^p(f_1(0), f_2(0))\right)}\Bigg\}}} - C_{\KW}\int_0^t A(s) \: ds\right)^{2}\Bigg\}.
    \end{multline}
\end{theorem}

The improvement of this stability estimate (\ref{eq:Iacobelli's W_p stability estimate}) of Theorem \ref{thm:Iacobelli's W_p} via $p$-kinetic Wasserstein distance compared to Loeper's stability estimate in $W_p$ (\ref{eq:Loeper's W_p stability estimate}) of Theorem \ref{thm:Loeper's W_p} lies in the order of magnitude of the time interval in which the two solutions are close to each other in Wasserstein distance. Indeed, if $W_p^p(f_1(0), f_2(0)) = \delta \ll 1$, then Loeper's stability estimate yields $W_p^p(f_1(t), f_2(t)) \lesssim 1$ for $t \in [0, \log(\abs{\log \delta})]$ while the kinetic stability estimates yields a better control of the time interval; $W_p^p(f_1(t), f_2(t)) \lesssim 1$ for $t \in [0, \sqrt{\abs{\log \delta}}]$.

\section{A new $L^p$-estimate via the Helmholtz-Weyl decomposition for $1 < p < +\infty$}

\subsection{Proof of the $L^p$-estimate}

\begin{proof}[Proof of Lemma \ref{thm:Delta U}]
    Let $[\phi] \in \dot{\D}(\X)$ be a quotient test function. Note that $\int \rho_1 - \rho_2 = 0$ as both $\rho_1$ and $\rho_2$ are probability measures, an integration by parts yields
    \begin{equation*}
        \int_{\X} \left[\phi\right] \left(\rho_1 - \rho_2\right) \: dx = \int_{\X} \phi \left(\rho_1 - \rho_2\right) \: dx = -\sigma\int_\X \nabla \phi \cdot \left(\nabla U_1 - \nabla U_2\right) \: dx.
    \end{equation*}
    First, we consider the torus case $\X = \Td$: We use the Helmholtz-Weyl decomposition given by Theorem \ref{thm:HW} to write any $\Rd$-valued test function $\Phi \in C_c^\infty(\Td)$ as $\Phi = \nabla\phi + g$, where $\nabla\phi \in G_{p'}(\Td)$ and $g \in H_{p'}(\Td)$ with $1/p + 1/p' = 1$. By definition, there is a divergence-free sequence of test functions $(g_k)_{k\in\N}$ whose $L^{p'}$-limit is $g$. By continuity of the force fields (see \cite[Lemma 3.2]{HKI2}), $\nabla U_1 - \nabla U_2 \in L^\infty(\Td)$, and in particular $\nabla U_1 - \nabla U_2 \in L^p(\Td)$. An integration by parts yields
    \begin{align*}
        \norm{\nabla U_1 - \nabla U_2}_{L^p(\Td)}
            &= \sup_{\norm{\Phi}_{L^{p'}(\Td)} \le 1} \left\{\int_{\Td} \Phi \cdot\left(\nabla U_1 - \nabla U_2 \right)\: dx \right\} \\ 
            &= \sup_{\norm{\Phi}_{L^{p'}(\Td)} \le 1} \left\{\int_{\Td} \nabla\phi \cdot\left(\nabla U_1 - \nabla U_2\right)\: dx - \lim_{k\to\infty} \int_{\Td} \divergence g_k \left(U_1 -  U_2\right)\: dx \right\} \\
            &= \sup_{\phi; \: \norm{\Phi}_{L^{p'}(\Td)} \le 1} \left\{\int_{\Td} \nabla\phi \cdot\left(\nabla U_1 - \nabla U_2\right)\: dx\right\}.
    \end{align*}
    Since the projection operator $P_{p'} : \Phi \,\mapsto\, g$ is bounded from $L^{p'}(\Td)$ to $H_{p'}(\Td)$, we have that
    \begin{equation*}
        \norm{\nabla\phi}_{L^{p'}(\Td)} = \norm{g - \Phi}_{L^{p'}(\Td)} \le \left(1 + C_{P_{p'}}\right)\norm{\Phi}_{L^{p'}(\Td)},
    \end{equation*}
    where $C_{P_{p'}}$ is the constant from (\ref{eq:HW proj}), and we set $C_{\HW} := 1 + C_{P_{p'}}$. Consider the larger set
    \begin{equation*}
        \left\{\norm{\nabla\phi}_{L^{p'}(\Td)}/C_{\HW} \le 1\right\} \supset \left\{\phi; \: \norm{\nabla\phi}_{L^{p'}(\Td)}/C_{\HW} \le \norm{\Phi}_{L^{p'}(\Td)} \le 1\right\}
    \end{equation*}
    that does not depend on $\Phi$ anymore. By replacing the supremum over this set, we obtain
    \begin{align*}
        \norm{\nabla U_1 - \nabla U_2}_{L^p(\Td)} &\le C_{\HW} \sup_{\norm{\nabla\phi}_{L^{p'}(\Td)} \le 1} \left\{\int_{\Td} \nabla\phi \cdot\left(\nabla U_1 - \nabla U_2\right)\: dx\right\} \\
            &= C_{\HW} \sup_{\norm{\nabla[\phi]}_{L^{p'}(\Td)} \le 1} \left\{\int_{\Td} [\phi]\left(\rho_1 - \rho_2\right) \: dx\right\}.
    \end{align*}
    We conclude by density of quotient test functions $\dot{\D}(\Td)$ in $\dot{W}^{1, p}(\Td)$ and by the definition of the strong dual homogeneous Sobolev norm.
    
    Second, we consider the whole space case $\X = \R^d$: Let $\varphi \in C_c^\infty(\R^d)$ be a test function and set $1/p + 1/p' = 1.$ We have that $U_1 - U_2 = \sigma G_d * (\rho_1 - \rho_2)$ a.e., where $G_d$ is the fundamental solution of the Laplace equation. Then, by symmetry of the convolution,
    \begin{align*}
        \norm{\partial_{x_j} U_1 - \partial_{x_j} U_2}_{L^p(\Rd)}
            &= \sup_{\norm{\varphi}_{L^{p'}(\Rd)} \le 1} \left\{\int_{\Rd} \varphi\, \partial_{x_j} G_d * \left(\rho_1 - \rho_2\right)\: dx \right\} \\
            &= \sup_{\norm{\varphi}_{L^{p'}(\Rd)} \le 1} \left\{\int_{\Rd} \left(\rho_1 - \rho_2\right) \partial_{x_j} G_d * \varphi\: dx \right\}.
    \end{align*}
    We denote $\phi = \partial_{x_j} G_d * \varphi$, and Calderon-Zygmund's inequality \cite[Theorem II.11.4]{Galdi2011} yields
    \begin{equation*}
        \norm{\nabla \phi}_{L^{p'}(\Rd)} \le C_{\HW} \norm{\varphi}_{L^{p'}(\Rd)}
    \end{equation*}
    for some constant $C_{\HW} > 0$ which only depends on $d$ and $p$, so that the supremum can be replaced by the larger set
        \begin{equation*}
            \left\{\phi \in \dot W^{1,p'}(\Rd);\: \norm{\nabla \phi}_{L^{p'}(\Rd)}/C_{\HW} \le 1\right\} \supset \left\{\varphi;\: \norm{\partial_{x_j} G_d * \varphi}_{L^{p'}(\Rd)}/C_{\HW} \le \norm{\varphi}_{L^{p'}(\Rd)} \le 1\right\}
        \end{equation*}
    independent of $\varphi$. We obtain
    \begin{align*}
        \norm{\partial_{x_j} U_1 - \partial_{x_j} U_2}_{L^{p'}(\Rd)} &\le C_{\HW}\sup_{\norm{\nabla \phi}_{L^{p'}(\Rd)} \le 1} \left\{\int_{\Rd} \phi \left(\rho_1 - \rho_2\right) \: dx\right\} \\
            &= C_{\HW}\sup_{\norm{\nabla \left[\phi\right]}_{L^{p'}(\Rd)} \le 1} \left\{\int_{\Rd} \left[\phi\right] \left(\rho_1 - \rho_2\right) \: dx \right\},
    \end{align*}
    and we conclude by definition of the strong dual homogeneous Sobolev norm.
\end{proof}

Before proving our new $L^p$-estimate, we first state the existence of an optimal transport map adapted to our context; 
\begin{theorem}[Gangbo-McCann {\cite[Theorem 1.2]{GangboMcCann1996}}]\label{thm:GM}
    Let $\rho_1, \rho_2$ be two probability measures on $\X$ that are absolutely continuous with respect to the Lebesgue measure. Then
    \begin{equation*}
        \displaystyle W_p(\rho_1, \rho_2) = \left(\inf_{T_\# \rho_1 = \rho_2} \bigg\{\int_{\X} \abs{x - T(x)}^p \: d\rho_1(x) \bigg\}\right)^{1/p},
    \end{equation*}
    where the infimum runs over all measurable mappings $T : \X \to \X$ that push forward $\rho_1$ onto $\rho_2$. Moreover, the infimum is reached by a $\rho_1(dx)$-almost surely unique mapping $T$, and there is a $\abs{\cdot}^p$-convex function $\psi$ such that $T = \Id_{\X} -  \left(\nabla\abs{\cdot}^p\right)^{-1} \circ \nabla\psi$, where we denote $\left(\nabla h^*\right)$ by $\left(\nabla h\right)^{-1}$ for a function $h$ with $h^*$ its Legendre transform.
\end{theorem}

\begin{proof}[Proof of Proposition \ref{thm:L^p estimate}]
    Let us denote by
    \begin{equation*}
        \rho_\theta := \left[\left(\theta - 1\right) T + \left(2 - \theta\right)\Id_{\X}\right]_{\#}\rho_1
    \end{equation*}
    the interpolant measure between $\rho_1$ and $\rho_2$, where $T$ is the optimal transport map of Theorem \ref{thm:GM}. Let $\phi \in C_c^\infty(\X)$ be a test function. By the properties of pushforwards of measures, it follows immediately that
    \begin{equation*}
        \int_{\X} \phi(x) \: d\rho_\theta(x) = \int_{\X} \phi\left((\theta - 1)T(x) + (2 - \theta)x\right)\: d\rho_1(x).
    \end{equation*}
    Lebesgue's dominated convergence theorem yields
    \begin{equation*}
        \frac{d}{d\theta} \int_{\X} \phi(x)\: d\rho_\theta(x) = \int_{\X} \nabla\phi\left((\theta - 1)T(x) + (2 - \theta)x\right)\cdot \left(T(x) - x\right)\: d\rho_1(x).
    \end{equation*}
    Now, by using Hölder inequality with respect to the measure $\rho_1$, we get
    \begin{align*}
        \frac{d}{d\theta} \int_{\X} \phi(x)\rho_\theta(x)\: dx &\le \left(\int_{\X}\abs{\nabla\phi\left((\theta - 1)T(x) + (2 - \theta)x\right)}^{p'}\: d\rho_1(x)\right)^{1/p'} \left(\int_{\X}\abs{x - T(x)}^p\: d\rho_1(x)\right)^{1/p} \\
        &= \left(\int_{\X}\abs{\nabla\phi(x)}^{p'}\: d\rho_\theta(x)\right)^{1/p'} \left(\int_{\X}\abs{x - T(x)}^p\: d\rho_1(x)\right)^{1/p}.
    \end{align*}
    The second term in the product is exactly $W_p(\rho_1, \rho_2)$ by Theorem \ref{thm:GM}. For the first one, thanks to \cite[Remark 8]{Santambrogio2009}, the $L^\infty$-norm of the interpolant is controlled by the one of the two measures;
    \begin{equation*}
        \norm{\rho_\theta}_{L^\infty(\X)} \le \max\left\{\norm{\rho_1}_{L^\infty(\X)}, \norm{\rho_2}_{L^\infty(\X)}\right\}.
    \end{equation*}
    Therefore,
    \begin{equation*}
        \frac{d}{d\theta} \int_{\X} \phi(x)\rho_\theta(x)\: dx \le \max\left\{ \norm{\rho_1}_{L^\infty(\X)}, \norm{\rho_2}_{L^\infty(\X)} \right\}^{1/p'}\left(\int_{\X} \abs{\nabla\phi(x)}^{p'} \: dx\right)^{1/p'}W_p(\rho_1, \rho_2).
    \end{equation*}
    Combining the above estimate with the fact that $\int \rho_2 - \rho_1 = 0$ and Fubini's theorem yields
    \begin{align*}
        \int_{\X} [\phi](x)\left(\rho_2(x) - \rho_1(x)\right)\: dx &= \int_{\X} \phi(x)\left(\rho_2(x) - \rho_1(x)\right)\: dx \\
        &= \int_1^2\left(\frac{d}{d\theta} \int_{\X} \phi(x)\rho_\theta(x)\: dx\right) d\theta \\
        &\le \max\left\{ \norm{\rho_1}_{L^\infty(\X)}, \norm{\rho_2}_{L^\infty(\X)} \right\}^{1/p'}\left(\int_{\X} \abs{\nabla[\phi](x)}^{p'} \: dx\right)^{1/p'}W_p(\rho_1, \rho_2).
    \end{align*}
    By restricting to quotient test functions $[\phi]$ such that
    $\norm{\nabla[\phi]}_{L^{p'}(\X)} \le 1$, we get the strong dual homogeneous norm so that
    \begin{equation*}
        \norm{\rho_1 - \rho_2}_{\dot{W}^{-1,p}(\X)} \le \max\left\{ \norm{\rho_1}_{L^\infty(\X)}, \norm{\rho_2}_{L^\infty(\X)} \right\}^{1/p'}W_p(\rho_1, \rho_2),
    \end{equation*}
    and we conclude by Lemma \ref{thm:Delta U}.
\end{proof}
\begin{remark}
    Loeper uses extensively that the optimal transport map $T$ is convex to rely on the gas internal energy theory developed by McCann (see \cite[Section 2]{McCann1997}) to estimate the $L^\infty$-norm of the interpolant. Here, we only have $|\cdot|^p$-convexity instead, while still the $L^\infty$-estimate on the interpolant is valid as showed, for instance, by Santambrogio \cite[Remark 8]{Santambrogio2009}. Loeper gives also an alternative proof \cite[Proposition 3.1]{Loeper2005} using the Benamou-Brenier formula \cite[Proposition 1.1]{Benamou2000}. The interpolant measure satisfies the continuity equation
    \begin{equation*}
        \partial_\theta\rho_\theta + \divergence_x (\rho_\theta v_\theta) = 0
    \end{equation*}
    for a vector field $v_\theta$ related to the Wasserstein distance through
    \begin{equation*}
        \int_{\X} \abs{v_\theta (x)}^2 \: d\rho_\theta (x) = W_2^2(\rho_1, \rho_2).
    \end{equation*}
    Differentiating both sides of Poisson's equation gives
    \begin{equation*}
        \Delta\partial_\theta U_\theta = -\partial_\theta \rho_\theta = \divergence_x(\rho_\theta v_\theta),
    \end{equation*}
    and integrating by parts against $\partial_\theta U_\theta$ itself as test function yields
    \begin{equation*}
        \int_{\X} \abs{\partial_\theta \nabla U_\theta}^2 \: dx = \int_{\X} \rho_\theta v_\theta \cdot \partial_\theta \nabla U_\theta \: dx
    \end{equation*}
    so that
    \begin{equation*}
        \norm{\partial_\theta\nabla U_\theta}_{L^2(\X)} \le \norm{\rho_\theta}_{L^\infty(\X)}^{1/2} W_2(\rho_1, \rho_2) \le \max\left\{\norm{\rho_1}_{L^\infty(\X)}, \norm{\rho_2}_{L^\infty(\X)}\right\}^{1/2} W_2(\rho_1, \rho_2),
    \end{equation*}
    and the conclusion follows after integrating over $\theta \in [1, 2]$.

    Even though there is a $W_p$ version of Benamou-Brenier formula \cite[Theorem 5.28]{Santambrogio2015}, there is no analog test function that allows to mimic this proof for $L^p$. 
\end{remark}

\section{Stability estimates revisited for Wasserstein-like distances}

\subsection{Loeper's estimate revisited}
The proof of Loeper's stability estimate in $W_p$ on the torus $\Td$ is similar to \cite[Theorem 3.1]{HKI2} using the new $L^p$-estimate (\ref{eq:L^p force field estimate}) from Proposition \ref{thm:L^p estimate}. It relies on the modified quantity (see \cite[Section 4]{GPI-WP})
\begin{align*}
    Q_p(t) &:= \int_{(\TdRd)^2} \abs{X_1(t,x,v) - X_2(t,y,w)}^p + \abs{V_1(t,x,v) - V_2(t,y,w)}^p \: d\pi_0(x,v,y,w) \\
        &= \int_{(\TdRd)^2} \abs{x-y}^p + \abs{v -w}^p \: d\pi_t(x,v,y,w),
\end{align*}
where $\pi_t \in \Pi(f_1(t), f_2(t))$ and $\pi_0$ is an optimal coupling that satisfies the marginal property;
\begin{multline}\label{eq:marginal property}
    \int_{\left(\TdRd\right)^2} \phi(x, v, y, w) \: d\pi_t(x,v,y,w) \\ = \int_{\left(\TdRd\right)^2} \phi\left(Z_1(t, x, v), Z_2(t, y, w)\right) \: d\pi_0(x, v, y, w), \quad \forall \phi \in C_b((\TdRd)^2).
\end{multline}
The last ingredients are the following estimates analog to (\ref{eq:WQ estimate}) relying on the definition of Wasserstein distance (see \cite[Lemma 3.6]{Loep}):
\begin{equation}\label{eq:WQ_p estimate}
     W_p^p(f_1(t), f_2(t)) \le Q_p(t), \quad W_p^p(\rho_{f_1}(t), \rho_{f_2}(t))  \le Q_p(t).
\end{equation}

\subsection{Kinetic Wasserstein distance revisited}
We prove the recent Iacobelli's stability estimate in $W_p$ both on the torus and on the whole space adapting the proof of \cite[Theorem 3.1]{Iacobelli2022}.

The proof relies on the modified quantity from the kinetic Wasserstein distance (see \cite[Section 4]{Iacobelli2022})
\begin{align*}
    D_p(t) &:= \frac{1}{p}\int_{(\XRd)^2} \lambda(t)\abs{X_1(t, x, v) - X_2(t, y, w)}^p + \abs{V_1(t, x, v) - V_2(t, y, w)}^p \: d\pi_0(x, v, y, w) \\
        &= \frac{1}{p}\int_{(\XRd)^2} \lambda(t)\abs{x - y}^p + \abs{v - w}^p \: d\pi_t(x, v, y, w),
\end{align*}
where $\lambda(t) = \abs{\log D_p(t)}^{p/2}$, whose specific choice of comes from optimisation considerations that will become apparent in the proof. One is able to bound
\begin{equation*}
    \dot{D}_p(t) \lesssim \frac{\dot{\lambda}(t)}{\lambda}D_p(t) + D_p(t) \left( \lambda^{1/p}(t) + \lambda^{-1/p}(t) \log\left(\frac{pD_p(t)}{\lambda(t)}\right)\right)
\end{equation*}
which can be rewritten as
\begin{equation*}
    \dot{D}_p(t) \lesssim D_p(t) \Big[ \lambda^{1/p}(t) + \lambda^{-1/p}(t) \log\left(D_p(t)\right)\Big]
\end{equation*}
recalling that $\lambda$ is a decreasing function and assuming that $\abs{\log(pD_p(t)/\lambda(t))} \lesssim \abs{\log D_p(t)}$ in some regime. The term inside the square brackets is now optimised considering $D_p(t)$ as a function of $\lambda(t)$.

We recall the Log-Lipschitz estimate on the force fields \cite[Lemma 3.2]{HKI2}, see also \cite[Lemma 3.3]{GPI-WP} and \cite[Lemma 8.1]{Majda_Bertozzi_2001}:
\begin{lemma}\label{thm:Log-Lip estimate}
    Let $U_i$ satisfy $\sigma\Delta U_i = \rho_{f_i} - 1,\: i = 1,2, \sigma = \pm 1$, on $\Td$ in the distributional sense. Then there is a constant $C > 0$ such that for all $x, y \in \Td,\: i =1, 2$, it holds
    \begin{equation}\label{eq:Log-Lip field estimate}
        \abs{\nabla U_i(x) - \nabla U_i(y)} \le C\abs{x - y}\log\left(\frac{4\sqrt{d}}{\abs{x - y}}\right)\norm{\rho_{f_i} - 1}_{L^\infty(\Td)}. 
    \end{equation}
\end{lemma}
\begin{lemma}\label{thm:Log-Lip estimate Rd}
    Let $U_i$ satisfy $\sigma\Delta U_i = \rho_{f_i},\: i = 1,2, \sigma = \pm 1$, on $\Rd$ in the distributional sense. Then there is a constant $C_d > 0$ such that
    \begin{equation*}
        \norm{\nabla U_i}_{L^\infty(\Rd)} \le C_d \left(\norm{\rho_{f_i}}_{L^1(\Rd)} + \norm{\rho_{f_i}}_{L^\infty(\Rd)}\right),
    \end{equation*}
    and for all $x, y \in \Rd$ with $\abs{x-y} < 1/e$, $i =1, 2$, it holds
    \begin{equation}\label{eq:Log-Lip field estimate Rd}
         \abs{\nabla U_i(x) - \nabla U_i(y)} \le C_d\abs{x - y}\log\left(\frac{1}{\abs{x - y}}\right)\left(\norm{\rho_{f_i}}_{L^1(\Rd)} + \norm{\rho_{f_i}}_{L^\infty(\Rd)}\right). 
    \end{equation}
    In particular, for all $x, y \in \Rd$,
    \begin{equation}\label{eq:full Log-Lip field estimate Rd}
        \abs{\nabla U_i(x) - \nabla U_i(y)} \le C_d\abs{x - y}\left(1 +\log^{-}(\abs{x-y})\right)\left(\norm{\rho_{f_i}}_{L^1(\Rd)} + \norm{\rho_{f_i}}_{L^\infty(\Rd)}\right),
    \end{equation}
    with $\log^{-}(s) := \max\{-\log(s),0\}$.
\end{lemma}

\begin{proof}[Proof of Theorem \ref{thm:Iacobelli's W_p}]
    \begin{align*}
    \dot{D}_p(t) = &\:\frac{1}{p} \int_{(\XRd)^2} \dot{\lambda}(t)\abs{X_1 - X_2}^p \: d\pi_0 \\
        &+ \int_{(\XRd)^2} \lambda(t)\abs{X_1 - X_2}^{p-2}(X_1 - X_2)\cdot(V_1 - V_2) \: d\pi_0 \\
        &+ \int_{(\XRd)^2} \abs{V_1 - V_2}^{p-2}(V_1 - V_2)\cdot(\nabla_xU_2(t, X_2) - \nabla_x U_1(t, X_1)) \: d\pi_0.
    \end{align*}
    The last two terms are estimated using H\"older's inequality with respect to the measure $\pi_0$, and we have
    \begin{multline}\label{eq:dot D decomposition}
        \dot{D}_p(t) \le \frac{1}{p} \int_{(\XRd)^2} \dot{\lambda}(t)\abs{X_1 - X_2}^p \: d\pi_0 \\ + \lambda^{1/p}(t)\left(pD_p(t)\right) +
        \left(pD_p(t)\right)^{1/p'} \left(\int_{(\XRd)^2} \abs{\nabla_x U_2(t, X_2) - \nabla_x U_1(t, X_1)}^p \: d\pi_0\right)^{1/p}.
    \end{multline}
    Recall the separation of the difference of force fields;
    \begin{equation*}
        \abs{\nabla_x U_2(t, X_2) - \nabla_x U_1(t, X_1)} \le \abs{\nabla_x U_2(t, X_2) - \nabla_x U_2(t, X_1)} + \abs{\nabla_x U_1(t, X_1) - \nabla_x U_2(t, X_2)},
    \end{equation*}
    whence
    \begin{equation}\label{eq:KW p-force field to T_1 T_2 estimate}
        \left(\int_{(\XRd)^2} \abs{\nabla_x U_2(t, X_2) - \nabla_x U_1(t, X_1)}^p \: d\pi_0\right)^{1/p} \le T_1(t) + T_2(t),
    \end{equation}
    where
    \begin{multline}\label{eq:KW T_1 T_2}
        T_1(t) := \left(\int_{(\XRd)^2} \abs{\nabla_x U_2(t, X_2) - \nabla_x U_2(t, X_1)}^p \: d\pi_0\right)^{1/p}, \\ T_2(t) := \left(\int_{(\XRd)^2} \abs{\nabla_x U_1(t, X_1) - \nabla_x U_2(t, X_1)}^p \: d\pi_0\right)^{1/p}. 
    \end{multline}
    First, consider the torus case $\X = \Td$: We estimate $T_1$ (\ref{eq:KW T_1 T_2}) using the non-decreasing concave function on $[0, +\infty)$ given by
    \begin{equation*}
        \Phi_p(s) :=
        \begin{cases}
            s\log^p\left(\frac{(4\sqrt{d})^p}{s}\right) & \text{if } s \le (4\sqrt{d}/e)^p, \\
            \left(\frac{4p\sqrt{d}}{e}\right)^p & \text{if } s > (4\sqrt{d}/e)^p,
        \end{cases}
    \end{equation*}
    together with the Log-Lipschitz estimate (\ref{eq:Log-Lip field estimate}) from Lemma \ref{thm:Log-Lip estimate} and (\ref{eq:A Iacobelli}) to get\footnote{Note that, since $\rho_{f_i}(t) \ge 0$ and $\norm{\rho_{f_i}(t)}_{L^\infty(\Td)} \ge 1$, then $\norm{\rho_{f_i}(t) - 1}_{L^\infty(\Td)} \le \norm{\rho_{f_i}(t)}_{L^\infty(\Td)} \le A(t)$ for i = 1, 2.}
    \begin{equation*}
        T_1(t) \le C A(t) \left(\int_{(\TdRd)^2} \Phi_p\left(\abs{X_1 - X_2}^p\right) \: d\pi_0\right)^{1/p}
    \end{equation*}
    provided that $\abs{X_1 - X_2}^p \le (4\sqrt{d}/e)^p$, but this is always the case since the distance between points in the torus cannot exceed $\sqrt{d}$. Thus by Jensen's inequality we have
    \begin{align*}
        T_1(t) &\le CA(t)\Bigg[\Phi_p\left(\int_{(\TdRd)^2} \abs{X_1 - X_2}^p \: d\pi_0\right)\Bigg]^{1/p} \le CA(t)\Bigg[\Phi_p\left(\frac{pD_p(t)}{\lambda(t)}\right)\Bigg]^{1/p}.
    \end{align*}
    Now, the considered regime becomes
    \begin{equation}\label{eq:regime 1}
        \frac{pD_p(t)}{\lambda(t)} \le \left(4\sqrt{d}/e\right)^p,
    \end{equation}
    so that
    \begin{equation}\label{eq:T_1 intermediate estimate}
        T_1(t) \le CA(t)\left(\frac{pD_p(t)}{\lambda(t)}\right)^{1/p}\Bigg[\abs{\log\left(\frac{pD_p(t)}{\lambda(t)}\right)} + p\log\left(4\sqrt{d}\right)\Bigg].
    \end{equation}
    We replace $\lambda(t) = \abs{\log D_p(t)}^{p/2}$, and consider yet another regime, now dictated by
    \begin{equation}\label{eq:regime 2}
        D_p(t) \le \frac{1}{e},
    \end{equation}
    so that $\abs{\log D_p(t)} \ge 1$. Note that this regime is compatible with the regime (\ref{eq:regime 1}) needed for the function $\Phi_p$ in the sense that if $pD_p(t) \le (4\sqrt{d}/e)^p $ holds, then $pD_p(t)/\abs{\log D_p(t)}^{p/2} \le pD_p(t) \le (4\sqrt{d}/e)^p$, and since $p/e \le (4\sqrt{d}/e)^p$, we can only consider the regime (\ref{eq:regime 2}). We estimate the logarithms in (\ref{eq:T_1 intermediate estimate}) in this new regime (\ref{eq:regime 2}) using an elementary inequality valid within this regime;
    \begin{equation}\label{eq:elementary ineq regime}
        \abs{\log\left(\frac{pD_p(t)}{\abs{\log D_p(t)}^{p/2}}\right)} \le 2\left(1 + \log p + \frac{p}{2}\right)\abs{\log D_p(t)},
    \end{equation}
    and set $C_p := 2(1+\log p + p/2)$. Hence (\ref{eq:T_1 intermediate estimate}) becomes
    \begin{equation}\label{eq:T_1 estimate}
        T_1(t) \le CA(t)\left(\frac{pD_p(t)}{\lambda(t)}\right)^{1/p}\Big[C_p\abs{\log D_p(t)} + p\log\left(4\sqrt{d}\right)\Big].
    \end{equation}
    We move to the estimation of $T_2$ (\ref{eq:KW T_1 T_2}). The $L^p$-estimate (\ref{eq:L^p force field estimate}) from Proposition \ref{thm:L^p estimate} yields
    \begin{equation}\label{eq:T_2 intermadiate estimate}
        T_2(t) \le C_{\HW}A(t)W_p(\rho_{f_1}(t), \rho_{f_2}(t))
    \end{equation}
    recalling (\ref{eq:A Iacobelli}). Since $(X_1(t), X_2(t))_{\#}\pi_0$ has marginals $\rho_{f_1}(t)$ and $\rho_{f_2}(t)$, we can estimate the Wasserstein distance between the densities by $D_p$ (see \cite[Lemma 3.6]{Loep}). More precisely, by the definition of the Wasserstein distance, we have
    \begin{align*}
    W_p^p(\rho_{f_1}(t), \rho_{f_2}(t)) &= \inf_{\gamma \in \Pi(\rho_{f_1}(t), \rho_{f_2}(t))} \int_{\TdTd} \abs{x-y}^p \: d\gamma(x,y) \\
        &\le \int_{\TdTd} \abs{x-y}^p \: d\Big[\left(X_1(t), X_2(t)\right)_\#\pi_0\Big](x,y) \\
        &= \int_{(\TdRd)^2} \abs{X_1(t,x,v) - X_2(t,y,w)}^p \: d\pi_0(x,v,y,w) \le \frac{pD_p(t)}{\lambda(t)}.
    \end{align*}
    We replace $\lambda(t) = \abs{\log D_p(t)}^{p/2}$ and the above estimate in (\ref{eq:T_2 intermadiate estimate}) to get
    \begin{equation}\label{eq:T_2 estimate}
        T_2(t) \le C_{\HW}A(t)\left(\frac{pD_p(t)}{\abs{\log D_p(t)}^{p/2}}\right)^{1/p}.
    \end{equation}
    Putting altogether estimates (\ref{eq:dot D decomposition}, \ref{eq:KW p-force field to T_1 T_2 estimate}, \ref{eq:T_1 estimate}, \ref{eq:T_2 estimate}) gives in the considered regime (\ref{eq:regime 2}) that
    \begin{multline*}
        \dot{D}_p(t) \le \frac{1}{p} \int_{(\TdRd)^2} \frac{-\dot{D}_p(t)}{ D_p(t)}\abs{\log D_p(t)}^{p/2-1}\abs{X_1 - X_2}^p \: d\pi_0 \\ +        \left(pD_p(t)\right)\Bigg[\sqrt{\abs{\log D_p(t)}} +  CA(t)\left(C_p\sqrt{\abs{\log D_p(t)}}(t) + \frac{p\log\left(4\sqrt{d}\right)}{\sqrt{\abs{\log D_p(t)}}}\right) + \frac{C_{\HW}A(t)}{\sqrt{\abs{\log D_p(t)}}}\Bigg].
    \end{multline*}
    If $\dot{D}_p \le 0$, then we do not do anything. Otherwise, then the first term in the above estimate is negative, and therefore
    \begin{equation*}\label{eq:sign dot KW bound}
        \dot{D}_p(t) \le \left(pD_p(t)\right)\Bigg[\sqrt{\abs{\log D_p(t)}} +  CA(t)\left(C_p\sqrt{\abs{\log D_p(t)}}(t) + \frac{p\log\left(4\sqrt{d}\right)}{\sqrt{\abs{\log D_p(t)}}}\right) + \frac{C_{\HW}A(t)}{\sqrt{\abs{\log D_p(t)}}}\Bigg].
    \end{equation*}
    Since the right-hand side is non-negative, independently of the sign of $\dot{D}_p$, this bound is always valid in the regime, and using that $\abs{\log D_p(t)} \ge 1$, together with $A(t) \ge 1$, we get
    \begin{equation*}
        \dot{D}_p(t) \le \tilde{C}_{\KW} A(t) D_p(t)\sqrt{\abs{\log D_p(t)}},
    \end{equation*}
    where $\tilde{C}_{\KW} := p \,\times \left[1 + C \times \left(C_p + p\log\left(4\sqrt{d}\right)\right) + C_{\HW}\right]$. Therefore,
    \begin{equation}\label{eq:D estimate}
        D_p(t) \le \exp\left\{-\left(\sqrt{\abs{\log D_p(0)}} - C_{\KW}\int_0^t A(s) \: ds\right)^{2}\right\},
    \end{equation}
    where $C_{\KW} := \tilde{C}_{\KW}/2$ depends only on $p$ and $d$. This implies in particular that (\ref{eq:regime 2}) holds if
    \begin{equation}\label{eq:regime 3}
        \sqrt{\abs{\log D_p(0)}} \ge C_{\KW}\int_0^T A(s) \: ds + 1.
    \end{equation}
    It remains to compare $D_p$ to the Wasserstein distance between the solutions in the regime (\ref{eq:regime 2}).
    By the definition of the Wasserstein distance, since $(Z_1(t), Z_2(t))_{\#}\pi_0$ has marginals $f_1(t)$ and $f_2(t)$, we have that
    \begin{align*}
    W_p^p(f_1(t), f_2(t)) &= \inf_{\gamma \in \Pi(f_1(t), f_2(t))} \int_{(\TdRd)^2} \abs{x-y}^p \: d\gamma(x,v,y,w) \\
        &\le \int_{(\TdRd)^2} \abs{x-y}^p + \abs{v-w}^p \: d\Big[\left(Z_1(t), Z_2(t)\right)_\#\pi_0\Big](x,v,y,w) \\
        &= \int_{(\TdRd)^2} \abs{X_1 - X_2}^p + \abs{V_1 - V_2}^p \: d\pi_0 \le pD_p(t).
    \end{align*}
    For the initial Wasserstein distance, since $\pi_0$ is optimal, we get
    \begin{equation*}
        D_p(0) \le \frac{1}{p}\abs{\log D_p(0)}\int_{(\TdRd)^2} \abs{x-y}^p + \abs{v-w}^p \: d\pi_0 = \frac{1}{p}\abs{\log D_p(0)}W_p^p(f_1(0), f_2(0)),
    \end{equation*}
    which we rewrite as
    \begin{equation*}
        \frac{D_p(0)}{\abs{\log D_p(0)}} \le \frac{1}{p}W_p^p(f_1(0), f_2(0)).
    \end{equation*}
    Note that, near the origin, the inverse of the function $s \mapsto s/\abs{\log s}$ behaves like $\tau \mapsto \tau\abs{\log \tau}$. In particular, there is a universal constant $c_0 > 0$ such that
    \begin{equation*}
        \frac{s}{\abs{\log s}} \le \tau\quad \text{for some $0 \le \tau \le c_0$} \quad \implies \quad s \le p\tau\abs{\log\tau}.
    \end{equation*}
    Hence for sufficiently small initial distance such that $W_p^p(f_1(0), f_2(0)) \le p c_0$, then
    \begin{equation*}
        D_p(0) \le W_p^p(f_1(0), f_2(0))\abs{\log\left(\frac{1}{p}W_p^p(f_1(0), f_2(0))\right)}.
    \end{equation*}
    Combining these bounds with (\ref{eq:D estimate}), and recalling (\ref{eq:regime 3}), this implies
    \begin{multline*}
        W_p^p(f_1(t), f_2(t)) \\ \le p\exp\Bigg\{-\left(\sqrt{\abs{\log\Bigg\{W_p^p(f_1(0), f_2(0))\abs{\log\left(\frac{1}{p}W_p^p(f_1(0), f_2(0))\right)}\Bigg\}}} - C_{\KW}\int_0^t A(s) \: ds\right)^{2}\Bigg\}
    \end{multline*}
    provided $W_p^p(f_1(0), f_2(0)) \le p c_0$ and
    \begin{equation*}
        \sqrt{\abs{\log\Bigg\{W_p^p(f_1(0), f_2(0))\abs{\log\left(\frac{1}{p}W_p^p(f_1(0), f_2(0))\right)}\Bigg\}}} \ge C_{\KW}\int_0^T A(s) \: ds + 1.
    \end{equation*}
    We conclude the proof of the torus case by \cite[Lemma 3.7 \& Remark 3.8]{Iacobelli2022}.

    Second, we consider the whole space case $\X = \Rd$: The only difference lies in the the separation of force fields. We have to estimate $T_1$ and $T_2$ defined in (\ref{eq:KW T_1 T_2}). We split $T_1$ in two integrals;
    \begin{multline*}
        T_1(t)^p = \left(\int_{\abs{X_1 - X_2} < 1/e} d\pi_0 + \int_{\abs{X_1 - X_2} \ge 1/e} d\pi_0\right) \Big[\abs{\nabla_x U_2(t, X_2) - \nabla_x U_2(t, X_1)}^p\Big] := I_1(t) + I_2(t).
    \end{multline*}
    On one hand, for $I_1$, using the Log-Lipschitz estimate (\ref{eq:Log-Lip field estimate Rd}) from Lemma \ref{thm:Log-Lip estimate Rd} and (\ref{eq:A Iacobelli}), we get
    \begin{align*}
        I_1(t) &\le C_d^p A^p(t) \int_{\abs{X_1 - X_2} < 1/e} \abs{X_1 - X_2}^p\log^p\left(\frac{1}{\abs{X_1 - X_2}^p}\right) \: d\pi_0 \\
            &\le C_d^p A^p(t) \int_{\abs{X_1 - X_2} < 1/e} \Phi_p(\abs{X_1 - X_2}^p) \: d\pi_0.
    \end{align*}
    Applying Jensen's inequality, we have
    \begin{equation*}
        I_1(t) \le C_d^p A^p(t) \Phi_p\left(\int_{\abs{X_1 - X_2} < 1/e} \abs{X_1 - X_2}^p \: d\pi_0\right) \le C_d^p A^p(t) \Phi_p\left(\frac{pD_p(t)}{\lambda(t)}\right).
    \end{equation*}
    On the other hand, for $I_2$, the  estimate (\ref{eq:full Log-Lip field estimate Rd}) from Lemma \ref{thm:Log-Lip estimate Rd} yields
    \begin{equation*}
        I_2(t) \le C_d^p A^p(t) \int_{\abs{X_1 - X_2} \ge 1/e} \abs{X_1 - X_2}^p \: d\pi_0 \le C_d^p A^p(t) \left(\frac{pD_p(t)}{\lambda(t)}\right).
    \end{equation*}
    Again, we impose the regime $D_p(t) \le 1/e$, with $\lambda(t) = \abs{\log D_p(t)}^{p/2}$, so that
    \begin{equation*}
        T_1(t) \le \left(I_1(t) + I_2(t)\right)^{1/p} \le 2^{1/p} C_d A(t) \left(\frac{pD_p(t)}{\lambda(t)}\right)^{1/p} \abs{\log\left(\frac{pD_p(t)}{\lambda(t)}\right) + p\log\left(4\sqrt{d}\right)}
    \end{equation*}
    becomes
    \begin{equation*}
        T_1(t) \le 2^{1/p} C_d A(t) \left(\frac{pD_p(t)}{\lambda(t)}\right)^{1/p} \left[C_p\abs{\log D_p(t)} + p\log\left(4\sqrt{d}\right)\right]
    \end{equation*}
    after using the elementary inequality (\ref{eq:elementary ineq regime}) valid within the considered regime. The estimation of $T_2$ (\ref{eq:KW T_1 T_2}) is again a direct consequence of the $L^p$-estimate (\ref{eq:L^p force field estimate}) from Proposition \ref{thm:L^p estimate};
    \begin{equation*}
        T_2(t) \le C_{\HW}A(t)W_p(\rho_{f_1}(t), \rho_{f_2}(t)).
    \end{equation*}
    From now on, the proof is the same as in the torus case $\X = \Td$ with
    \begin{equation*}
        C_{\KW} := p \,\times \left[1 + 2^{1/p}C_d \times \left(C_p + p\log\left(4\sqrt{d}\right)\right) + C_{\HW}\right].
    \end{equation*}
\end{proof}

\noindent \textbf{Acknowledgement:} The authors would like to thank the anonymous referees for their useful and detailed comments, which improved the presentation of the paper. The second author also acknowledges partial financial support from the Dutch Research Council (NWO): project number OCENW.M20.251.

\printbibliography
\end{document}